\documentclass{amsart}
\usepackage{relsize}
\usepackage{hyperref}

 \newtheorem{thm}{Theorem}[section]
 \newtheorem{cor}[thm]{Corollary}
 \newtheorem{lem}[thm]{Lemma}
 \newtheorem{prop}[thm]{Proposition}
 \theoremstyle{definition}
 \newtheorem{defn}[thm]{Definition}
 \theoremstyle{remark}

 \numberwithin{equation}{section}

\usepackage{color}
\usepackage{tikz,tikz-3dplot}

\newcommand{\mk}[1]{{\color{blue}#1}}
\renewcommand{\mk}[1]{#1}
\newcommand{\llm}[1]{{\color{black}#1}}

\newcommand{\R}{\mathbb{R}}

\newcommand{\ee}{\varepsilon}
\newcommand{\1}{{\mathbf 1}}

\DeclareMathOperator{\myspan}{span}
\newcommand{\cH}{{\mathcal H}}

\newcommand{\dd}{\mathrm{d}}
\DeclareMathOperator{\co}{co}
\DeclareMathOperator{\diam}{diam}

\begin{document}

%
%
%
%
%
%
%
%
%

\title[Stability results for barriers]{Stability for barriers of $n$-dimensional\\ convex bodies with surface area close to\\   Jones' bound}

\author[M.~Kiderlen]{Markus Kiderlen, ORCID 0000-0003-2858-6659}
\address{%
Aarhus University, Aarhus, Denmark}
\email{kiderlen@math.au.dk}




\date{May 4, 2026}

\begin{abstract}
Let $K$ be a convex body (a non-empty compact convex set) in $n$-dimensional Euclidean space.  
A set $B$  is called \emph{a barrier} (or an `opaque set') for $K$ if every line that intersects $K$, also intersects $B$. Although this concept was introduced more than a century ago, the barrier  with minimal surface area for a given set $K$ is still unknown, even in the two-dimensional case. 
A classical lower bound by Jones states that the surface area $S(B)$ of a sufficiently regular barrier $B$ is at least $S(\partial K)/2$, half the surface area of the boundary of $K$. We will extend a known stability version for $n=2$ to arbitrary dimensions: if  $S(B)-S(\partial K)/2$ is small, then the orientation measure of $B$ is close to the surface area measure of a symmetrization of $K$. For instance, if $K$ is the unit cube in 3D, most of the points of a barrier with surface area close to $3$ must have almost axis parallel normals. 

One of the main contributions of the paper is the new concept of weak barriers, 
which only encodes orientation information of a barrier, disregarding  
the relative positions of its parts. We characterize weak barriers geometrically in terms of the convexification of $B$. Convex geometric tools then allow one to quantify the above mentioned stability for weak barriers in all dimensions.  
\end{abstract}

\maketitle

\section{Introduction and main results}
Let $K$ be a convex body in $\R^n$, $n\ge 2$,  i.e.~$K\subset\R^n$ is a non-empty compact convex set. 
A set $B \subseteq \R^n$ is called a \emph{barrier} (or an \emph{opaque set})  for $K$ if any line that intersects $K$ also intersects $B$.  Barriers need not be connected nor are they required to be subsets of $K$.

Clearly, the boundary of $K$ is a barrier for $K$, so there are barriers with finite surface area, and one may ask for the  \emph{smallest} barrier for a given $K$ in terms of surface area. In the case of planar polygons, this question was already asked more than a century ago by Mazurkiewicz \cite{maz}. 
Even for very  simple sets $K\subset \R^2$ such as a square, a disc or an equilateral triangle the length of the shortest barrier
(or the infimum over all such lengths, in case there is no smallest barrier) is unknown. 
The problem of finding a shortest barrier has been an active area of research for decades: besides theoretical papers, e.g.~\cite{Faber,Kawohl,pach,Izumi}, there are many publications dealing with computational issues, such as \cite{Dumitrescu,Dumitrescub}. In addition, due to its simplicity, it is a common theme in recreational mathematics, see e.g.~\cite{Finch,MGardner,stewart}. 

\mk{Most of the literature deals with the planar situation. 
	The extension to higher dimensions, more specifically the problem of finding the barrier for the unit cube in $\R^3$ with minimal surface area, is often attributed to Gardner's publication \cite{MGardner} in 1990. 
	But already  in 1984 Faber \emph{et al.}~\cite{Faber} pose the corresponding question for barriers of the unit ball in three dimensions, and in 1986 Faber \& Mycielski \cite{FaberMycielski1986} stated explicitly the above definition of barriers for $n$-dimensional sets. }
	
The best theoretical results either yield upper bounds for the minimal length  by constructing short barriers (cf.~Fig.~\ref{bla} (left) for the unit square in the plane) or establish lower bounds. The present paper is related to the construction of lower bounds. 

To explain our contribution properly, we first state the problem in more concise terms in the two-dimensional setting to maintain the elementary flair of the question.  In this introduction, when considering $n=2$, we will restrict considerations to \emph{rectifiable} barriers in the sense of \cite{pach}, i.e.~barriers that can be written as unions of at most countably many rectifiable curves. We therefore are interested in 
\begin{equation}\label{eq:b(K)}
b(K)=\inf\{ |B|: B \text{ is a rectifiable barrier for } K\}, 
\end{equation}
where $|\cdot|$ denotes length. \mk{For barriers $B$ of a convex body $K\subset\R^n$, $n\ge 2$, one naturally would measure the surface area of $B$ instead of the length, where certain regularity must be imposed on $B$. Not even the existence of a barrier with minimal surface area is known, but Kawohl \cite{Kawohl} showed that such a minimizer exists in the $n$-dimensional setting,  if one restricts considerations to interior barriers ($B\subset K$) with a fixed bound on the number of connected components. }

We return to the two-dimensional situation. 
For any rectifiable barrier $B\subset\R^2$ there is a \emph{straight} barrier (a finite union of line segments) with a length that is arbitrarily close to $|B|$, see \cite[Lemma 4]{pach}. Thus, the infimum in \eqref{eq:b(K)} can equivalently be taken over  straight barriers only, and we can work with straight barriers in the following. 

For a long time 
the best lower bound was the one determined by Jones  in 1962 
(cf.~\cite{jones}, originally only for the unit square), stating that the length of any rectifiable barrier $B$ of a convex body $K\subset \R^2$ is at least half the perimeter of $K$, i.e. \begin{equation}
	\label{eq:jones}
	|B|\ge \tfrac12 |\partial K|,
\end{equation} where  
$\partial K$ is the boundary of $K$. Virtually all methods to establish lower bounds start with Jones' bound and try to improve it, exploiting properties of a hypothetical straight barrier $B$ with a length `too close' to  $\frac12 |\partial K|$, to obtain a contradiction. 

Typically, these proofs use at some point what might be called `orientation information' of the segments of $B$, disregarding their mutual positions and the positions with respect to $K$. We formalize this aspect of the barrier defining the following concept: A set $B\subset \R^2$ is called a \emph{weak barrier} for $K$ if, for any line $g$ through the origin, the total projection length of $B$ onto $g$ (with multiplicities) is not smaller than the projection length of $K$ onto $g$. Any barrier for $K$ is a weak barrier for $K$. Indeed, if $B$ is a barrier for $K$, then all lines orthogonal to $g$ that hit $K$ must hit $B$, so the geometric projection of $B$ (which ignores multiplicities) must contain the projection of $K$, implying the claim. 
But weak barriers need not be barriers; consider e.g.~the set $B=\tfrac12 \partial K$, obtained by down-scaling $\partial K$ to half of its size.
If we translate parts or all of a weak barrier $B$ for $K$, for instance some of the line segments it consists of, the result will  again be a weak barrier for $K$. 

\newcommand{\myblue}{black}
\newcommand{\myred}{black}
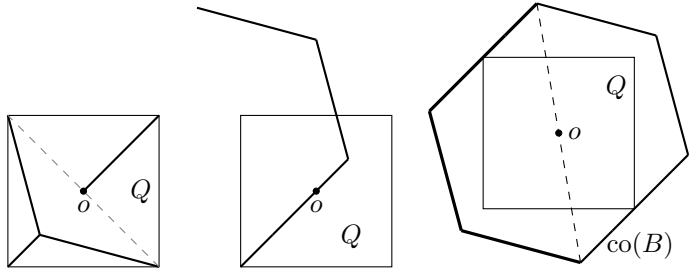
\begin{figure}
	\centering 
	\begin{tikzpicture}[scale=2]
		\draw[black] (0,0) rectangle (1,1);
		
		\coordinate (S) at (0.2113, 0.2113);
		
		\draw[thick, \myblue] (0,1) -- (S);
		\draw[thick, \myblue] (0,0) -- (S);
		\draw[thick, \myblue] (1,0) -- (S);
		
		\draw[thick, \myblue] (1,1) -- (0.5, 0.5);
		
		\draw[dashed, gray] (1,0) -- (0,1);
		\draw[black] (0.75,0.5) node[anchor=west,black] {$Q$};
		
		
		\filldraw[black] (0.5,0.5) circle (0.02) node[anchor=north] {$o$};
	\end{tikzpicture}
	\quad 
	\begin{tikzpicture}[scale=2]
		\draw[black] (0,0) rectangle (1,1);
		\filldraw[black] (0.5,0.5) circle (0.02) node[anchor=north] {$o$};
		\draw[black] (0.6,0.2) node[anchor=west,black] {$Q$};
		
		\coordinate (S) at (0.2113, 0.2113);
		
		\begin{scope}[shift={(0.5,0.5)}]	
			\draw[thick, \myblue] (0,1) -- (0.2113, 0.2113);
		\end{scope}
		\begin{scope}[shift={(0,0)}]
			\draw[thick, \myblue] (0,0) -- (0.2113, 0.2113);
		\end{scope}
		\begin{scope}[shift={(-0.5,1.5)}]
			\draw[thick, \myblue] (1,0) -- (0.2113, 0.2113);
		\end{scope}
		\begin{scope}[shift={(-0.2887,-0.2887)}]
			\draw[thick, \myblue] (1,1) -- (0.5, 0.5);
		\end{scope}
		
		
	\end{tikzpicture}
	\quad 
	\begin{tikzpicture}[scale=2]
		\draw[black] (0,0) rectangle (1,1);
		\filldraw[black] (0.5,0.5) circle (0.02) node[anchor=west] {$o$};
		
		\coordinate (S) at (0.2113, 0.2113);
		
		\begin{scope}[shift={(0.6444,-0.3557)}]
			\draw[dashed] (0,0) -- (-0.2887,1.7113);
			\begin{scope}[shift={(0.5,0.5)}]
				\draw[thick, \myblue] (0,1) -- (0.2113, 0.2113);
			\end{scope}
			\begin{scope}[shift={(0,0)}]
				\draw[thick, \myblue] (0,0) -- (0.2113, 0.2113);
			\end{scope}
			\begin{scope}[shift={(-0.5,1.5)}]
				\draw[thick, \myblue] (1,0) -- (0.2113, 0.2113);
			\end{scope}
			\begin{scope}[shift={(-0.2887,-0.2887)}]
				\draw[thick, \myblue] (1,1) -- (0.5, 0.5);
			\end{scope}	
		\end{scope}
		\begin{scope}[shift={(0.3556,1.3557)},rotate=180]
			\begin{scope}[shift={(0.5,0.5)}]
				\draw[very thick, \myred] (0,1) -- (0.2113, 0.2113);
				\draw (-0.9, 1.1) node[anchor=west,black] {$\co (B)$};
			\end{scope}
			\begin{scope}[shift={(0,0)}]
				\draw[very thick, \myred] (0,0) -- (0.2113, 0.2113);
			\end{scope}
			\begin{scope}[shift={(-0.5,1.5)}]
				\draw[very thick, \myred] (1,0) -- (0.2113, 0.2113);
			\end{scope}
			\begin{scope}[shift={(-0.2887,-0.2887)}]
				\draw[very thick, \myred] (1,1) -- (0.5, 0.5);
			\end{scope}	
		\end{scope}
		
		\draw[black] (0.75,0.8) node[anchor=west,black] {$Q$};
		
	\end{tikzpicture}

	\caption{\label{bla} Left: The best known barrier $B$ for the centered unit square $Q=[-\frac12,\frac12]^2$ consisting of the Steiner tree connecting the three lower left points and an additional segment in the upper right. Its length is approximately $2.64$. 
		Middle: first step of the construction of its convexification; see main text. Right: Second step in this construction. Since $B$ is an opaque set, it is also weakly opaque, so $Q\subset \co B$ in accordance with Proposition \ref{prop:1}. The fact that the boundary of $\co(B)$ hits $Q$ in one pair of antipodal points 
		shows that there is exactly one projection direction where the total projection length of $B$ is the same, no matter if we consider it with or without multiplicities. 
	}
\end{figure}

We will now outline a well-known convex-geometric construction that characterizes the orientation information captured in the notion of weak barriers. 
To fix ideas, consider a finite union $B$ of line segments in $\R^2$ and define its 
\emph{convexification} $\co(B)$. Convexifications of sets have been used by 
Pach \cite{Pach78} (to maximize the area of  planar regions enclosed by line segments) and Fáry \& Makai \cite{FaryMakai} (in the context of an isoperimetric problem, mostly $n=2$), Böröczky et al.~\cite{BeEtAl} (to find, among others,  maximally enclosing polytopes) and Weil \cite{weil,weil95} (in applications to stochastic geometry). 
It is important to notice that $\co(B)$ is different from the usual convex hull of $B$. The construction  of $\co(B)$  goes as follows: 
order the segments of $B$  according to increasing angles in $[0,\pi)$ with respect to the horizontal axis and fit these segments together 
to form a convex polygonal curve. Translate this curve such that the midpoint of its endpoints is the origin.   The union of this curve  and its reflection at the origin is the boundary of a convex polygon, the convexification $\co(B)$. This construction is illustrated in Fig.~\ref{bla}. Note that the present definition of convexification always yields a convex set, while some authors reserve the notion of  convexification  to the boundary of $\co(B)$. We will later see (for $n=2$), that the set $\co(B)$ is up to translation the Minkowski sum of the line segments constituting  $B$. 

The first result compares this convex set with a symmetrized version of $K$, the so-called 
\emph{%
	Blaschke body $\nabla K$} of $K$, see Sect.~\ref{sect:prelim} for a definition. In the present, two-dimensional case, 
$\nabla K=\tfrac12(K+(-K))$ is simply the \emph{central symmetral of $K$}, see \cite[Def.~3.2.2]{Gardner}. In higher dimensions, Blaschke body and central symmetral of $K$ are typically different.

\begin{prop}\label{prop:1}
	Assume that  $B\subset \R^2$ consists of finitely many line segments and  that  $K\subset \R^2$ is a convex body with interior points.  Then 	
	\[
	B \text{ is a weak barrier for } K \iff \nabla K\subset \co (B). 
	\]
	
	In particular, if $B$ is a weak barrier for $K$, its length satisfies  Jones' bound \eqref{eq:jones}. 
\end{prop}
Although the paper will establish an analogous result in much higher generality, we outline the proof, as it illustrates the underlying method in elementary terms. The generalization of this construction will be described later and  shows that the characterization remains true for more general $B$, more specifically for sets $B$ that are suitably rectifiable in a geometric measure theoretic sense.  
\begin{proof}[Proof of Prop.~\ref{prop:1}.]
	If $B$  is a straight weak barrier for $K$, the boundary of $\co(B)$ consists of 
	duplicated translates of the segments in $B$, so the total projection length of $\partial \co(B)$ (with multiplicities) on a line $g$ is at least twice the projection length of $K$ on $g$. In other words, the 	projection length of the convex body $\co(B)$ is at least the projection length of $K$, the latter being equal to the projection length of $\nabla K$. Since the projections of the origin-symmetric convex sets 
	$\co(B)$ and $\nabla K$ are line segments centered at the origin, we see that all projections of $\co(B)$ contain the corresponding projections of $\nabla K$, implying the inclusion. The converse statement follows with similar arguments. 
	
	Since the perimeter map is monotone with respect to set inclusions of convex bodies, we have 
	for a weak barrier $B$, 
	\[
	2|B|=|\partial (\co (B))|\ge |\partial(\nabla K)|=|\partial K|,  
	\]
	establishing Jones' bound.  
\end{proof}

We note that if $K$ already is origin-symmetric, then $\nabla K=K$. Hence, 
a typical application of Proposition \ref{prop:1} is the statement that $B$ is a weak barrier for the centered unit square $Q=[-\tfrac12,\tfrac12]^2$ if and only if $Q\subset \co(B)$; see Figure \ref{bla}. 
It is intuitively clear that 
the orientations of the segments in the boundary of $\co(B)$ (and hence, of the segments in $B$) must be very close to those of 
the segments in the boundary of $K$ if $K$ and $\co(K)$ almost have the same perimeter. More formally, we can compare the 
orientation measures $S^*(B,\cdot)=S(\co(B),\cdot)$ and  $S(K,\cdot)$, defined here on $S^1$ as weighted sums of Dirac measures on unit normals weighted by the lengths of the corresponding line segments, see Sect.~\ref{sect:prelim} for the general definition also in higher dimensions. This paper is heavily inspired by Steinerberger's   stability result \cite{steinerberger} for barriers in the planar case. In \cite{steinerberger}, the
distance of the above-mentioned orientation measures is expressed in terms of the homogeneous Sobolev space norm $\|\cdot\|_{\dot{H}^{-2}}$ and bounded by a power of the deficit $|B|-\frac12|\partial K|$ in Jones' bound. We prefer to work with the usual \emph{bounded-Lipschitz metric} $d_{\mathrm{bL}}$, defined in Sect.~\ref{sect:prelim}, and show in the planar case the bound
\begin{equation}\label{eq:2Dstability}
	d_{\mathrm{bL}}\big(S(\nabla K,\cdot),S^*(B,\cdot)\big)\le c\,\big(|B|-\tfrac12|\partial K|\big)^{\frac1{\llm{4}}-\ee},  
\end{equation}
where the constant $c$ only depends on the radius of a ball in $K$, on some $\ee>0$ and on $|B|$, see below.  Comments at the end of this introduction will also give an interpretation of \eqref{eq:2Dstability}.  \medskip 

All this can be extended to higher dimensions. 
\mk{Before describing our findings in $n$-dimensional Euclidean space, it should be noted that almost no suggestions for explicit  barriers (apart from trivial ones)  exist in the literature when the underlying convex body is a subset of $\R^n$, $n\ge 3$.  Brakke \cite{Brakke}
constructed a barrier  for the unit cube $[-1/2,1/2]^3$ with surface area approximately equal to $4.2324$. His approach is based on the best known barrier of $Q\subset\R^2$ depicted in Fig.~\ref{bla}. Writing $B'$ for this $2$-dimensional barrier, Brakke starts out considering its extension to three dimensions $B''=B'\times[-1/2,1/2]$ and constructs the barrier $B=B''\cup B_+\cup B_-$, where $B_\pm=[-1/2,1/2]^2\times \{\pm1/2\}$ are the `lid' and the `floor' of the unit cube. The set $B''$ is a barrier for the unit cube of approximate length $1\cdot 2.64+2=4.64$, but Brakke diminishes its surface area further by applying a minimal surface solver to it. 
This yields the already mentioned best known barrier for the unit cube with approximate length $4.2324$. 
In the same paper, Brakke considers barriers for the unit sphere $S^2$ in $\R^3$ and shows that a sequence $(B_m)$  of (not necessarily very small) barriers can be constructed in a way such that their set limit $B$ in the Hausdorff metric is still a barrier for $S^2$ but has a larger surface area then the limit of the surface areas of the sets $B_m$.  

In \cite{Asimov}, minimal barriers on manifolds are considered. These are sets with minimal (density of) measure, hitting every geodesic of the  manifold, for instance the 2-sphere or the flat torus. Also  Euclidean spaces are considered shortly, but the observations in \cite{Asimov} are not relevant for the current discussion, since in this publication, \emph{all} test lines must have common points with the barrier, not only those hitting an underlying set $K$. 

\tdplotsetmaincoords{70}{120}

\begin{center}
	\begin{figure}[ht]
		\begin{tikzpicture}[tdplot_main_coords, scale=4]
			
			\pgfmathsetmacro{\zhalf}{0.07}
			
			\coordinate (A) at (-0.5,-0.5,-\zhalf); 
			\coordinate (B) at (0.5,-0.5,-\zhalf);  
			\coordinate (C) at (0.5,0.5,-\zhalf);   
			\coordinate (D) at (-0.5,0.5,-\zhalf);  
			\coordinate (E) at (-0.5,-0.5,\zhalf);  
			\coordinate (F) at (0.5,-0.5,\zhalf);   
			\coordinate (G) at (0.5,0.5,\zhalf);    
			\coordinate (H) at (-0.5,0.5,\zhalf);   
			
			\tikzset{facefill/.style={fill opacity=0.4}}
			
			\draw[gray!60 ] (A) -- (B);
			\draw[gray!60 ] (A) -- (D);
			\draw[gray!60 ]  (A) -- (E);
			
			\fill[gray!70!black, facefill] (A) -- (B) -- (C) -- (D) -- cycle;

			\fill[gray!85!black, facefill] (A) -- (B) -- (F) -- (E) -- cycle;
			\fill[gray!80!black, facefill] (A) -- (D) -- (H) -- (E) -- cycle;
			\fill[black!50!gray, facefill] (B) -- (C) -- (G) -- (F) -- cycle;
			\fill[black!90!gray, facefill] (C) -- (D) -- (H) -- (G) -- cycle;
			
			\fill[gray!20, facefill] (E) -- (F) -- (G) -- (H) -- cycle;
			
			
			\draw[black] (E) -- (F) -- (G) -- (H) -- cycle;
			
			\draw[black] (B) -- (C) -- (D);
			
			\draw[black] (B) -- (F);
			\draw[black] (C) -- (G);
			\draw[black] (D) -- (H);
			\draw (0,0,-0.2) node{$B_\varepsilon$};
			\draw (0,-0.2,0) node{$K_\varepsilon$};
			
		\end{tikzpicture}
		\caption{
			\label{fig:new} \mk{Example showing that Jones bound is sharp in three-dimensional space when the bound may only depend on the surface area of $K$. The set $B_\ee$ (dark gray) is the boundary of the `waver' $K_\ee=c_\ee([0,1]^2\times [0,\ee])$ (light gray) without the `lid', see main text. The constant is chosen as $c_\ee=1/\sqrt{1+2\ee}$ to assure that the surface area $S(K_\ee)=2$ does not depend on $\ee>0$. 
			$B_\ee$ is a barrier for $K_\ee$ with $S(B_\ee)=c_\ee^2(1+4\ee)$. The deficit in Jones' bound 
			$S(B_\ee)-\tfrac12 S(K_\ee)= 2c_\ee^2\ee$ can be arbitrarily small. }
		}
	\end{figure}
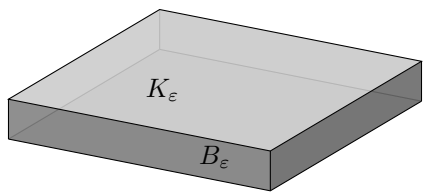
\end{center}

For $n\ge 3$, the only known \emph{lower} bound for $b(K)$ is Jones' bound, an extension of \eqref{eq:jones} to higher dimensions. 
The regularity assumption of the set $B$ can actually be weakened for $n=2$ and generalized to $\R^n$, see Section \ref{sect:prelim} for a definition of $(n-1)$-dimensional rectifiability. Jones bound states that if $B$ is an $(n-1)$-dimensional rectifiable barrier for the convex body $K\subset \R^n$, then 
\begin{equation}\label{eq:jones_n} 
	S(B)\ge \tfrac12 S(\partial K), 
\end{equation} 
where $S(\cdot)$ denotes surface area (i.e.~$(n-1)$-dimensional Hausdorff measure). The proof of \eqref{eq:jones_n}  is almost literally the same as in the two-dimensional case: any generic line $g$  in $\R^n$ hits the boundary of $K$ either not at all or twice. In the latter case, $g$ hits the barrier, so we have $\#(B\cap g)\ge \tfrac12 \#(\partial K \cap g)$. Integrating both sides of this inequality 
with respect to a motion invariant measure on the affine Grassmannian of lines in $\R^n$ yields \eqref{eq:jones_n} in view of \cite[3.16 Crofton's formula]{Morgan}. 

Just like in the two-dimensional case, \eqref{eq:jones_n}  is sharp if the lower bound only may depend on $S(\partial K)$. 
Consider for instance the `waver' 
\[
K_\ee=c_\ee\big([0,1]^{n-1}\times [0,\ee]\big),
\] which can be seen as a scaled version of the unit cube in $\R^n$, where the last coordinate is compressed to thickness $c_\ee\ee$, with $\ee>0$.  We choose 
\[
c_\ee=(1+(n-1)\ee)^{-\frac{1}{n-1}},
\]
which gives $S(\partial K_\ee)=c_\ee^{n-1}(2+2(n-1)\ee)=2$. The set $B_\ee=\partial K_\ee\setminus \big(c_\ee ((0,1)^{n-1}\times \{\ee\})\big)$ is clearly a barrier for $K_\ee$, see Fig.~\ref{fig:new} for the case $n=3$. But since  $S(B_\ee)=c_\ee^{n-1}(1+2(n-1)\ee)$, the deficit in Jones' bound is
\[
S(B_\ee)-1=
S(B_\ee)-\tfrac12 S(\partial K_\ee)=(n-1)c_\ee^{n-1}\ee\to 0,
\] 
as $\ee\to 0_+$. This example shows the following: If $b\ge 0$ is such that $S(B)\ge b$ for all barriers of convex bodies with given surface area $S>0$, then $b\le \frac12 S$, so Jones' bound is optimal uniformly in $K$. 

However, for fixed $K\subset \R^n$ one would expect that Jones' bound can be improved upon.  
For $n=2$, this expectation has been confirmed for all barriers of full-dimensional convex bodies apart from triangles in \cite[Thm.~3]{pach}, see also the penultimate paragraph in this section. It also holds for the equilateral triangle \cite{Izumi}. 
One would strongly expect that this is true in all dimensions. In order to show that for every full-dimensional convex body $K\subset\R^n$ there is a 
$\delta=\delta(K)>0$ such that $S(B)\ge \tfrac12 S(\partial K)+\delta$ holds for all barriers of $B$,  one strategy could be the following: assuming a hypothetical barrier $B$ with surface area (too) close to Jones' bound, which properties must $B$ have that might lead to a contradiction? The stability result  for $n$-dimensional barriers, that we will present now, is one way of deriving such additional properties.

}

We will therefore now describe how our findings for the planar case generalize to higher dimensions, where we will use the usual extension of the convexification as outlined in Section \ref{sec.mainresults}. The main difference to the planar case is that Proposition \ref{prop:1} 
no longer holds, but requires a transformation $\Pi$ of both convex bodies, involving the  projection body 
$\Pi K$ of $K$, encoding all projection areas of $K$ on hyperplanes. We have $\Pi K=\Pi(\nabla K)$.  

\begin{thm} \label{thm:main}
	Let a convex body $K\subset \R^n$ with interior points and an $(n-1)$-dimensional rectifiable set $B\subset\R^n$ be given. Then 
	\begin{equation}\label{eq:thm1}
		B \text{ is a weak barrier for } K \iff \Pi K\subset \Pi\big(\co (B)\big). 
	\end{equation}
	
	In particular, if $B$ is a weak barrier for $K$, it satisfies  Jones' bound \eqref{eq:jones_n}.
\end{thm}

If $K$ is an origin-symmetric convex body in $\R^2$, its projection body coincides with a rotation of $K$ with angle $\pi/2$ at the origin (no matter which orientation is chosen). Thus, the characterization of a weak barrier in Theorem \ref{thm:main} coincides with the simpler one in Proposition \ref{prop:1} for $n=2$. But in higher dimensions the connection between $\Pi K$ and $K$ is more complicated. We will see that already in dimension $n=3$, the condition in \eqref{eq:thm1} can hold even if $\nabla K\not\subset \co (B)$. Nevertheless, a stability result can be derived from Theorem \ref{thm:main} strengthening Jones' deficit inequality
$S(B)-\frac12 S(\partial K)\ge 0$. We will write $rB^n$ for the Euclidean ball of radius $r>0$ centered at the origin.

\begin{thm} \label{thm:mainstabilityLP}
	Let a convex body $K\subset \R^n$ with $rB^n\subset K$ for some $r>0$, and an $(n-1)$-dimensional rectifiable weak barrier $B$ for $K$ be given. 
	
	Then, for any $\ee>0$ there is a constant $c=c(\ee,n,r,S(B))$ such that 
	\begin{equation}\label{eq:mainStab}
		d_{\mathrm{bL}}\big(S(\nabla K,\cdot),S^*(B,\cdot) \big)\le c\, (S(B)-\tfrac12 S(\partial K))^{\mk{\frac{\llm{n+1}}{n(n+4)}}-\ee}. 
	\end{equation}	
\end{thm}
This theorem, specialized to the two-dimensional case, yields \eqref{eq:2Dstability}. \medskip

If $K$ is a convex polytope (i.e.~the convex hull of finitely many points) its boundary consists of finitely many $(n-1)$-dimensional convex polytopes, the so-called \emph{facets}.  Its Blaschke body $\nabla K$ is an origin-symmetric polytope with all its facets parallel to facets of $K$.
The above theorem implies that a weak barrier with surface area close to Jones' bound must have orientations that are roughly parallel to facets of $\nabla K$, and hence those of $K$. To quantify this statement, fix $0<\beta<\pi/4$, let $V$ be the set of all unit vectors $u$ such that either $u$ or $-u$ is an outer facet unit normal of $K$, and define the set 
\[
J_{\beta}=\bigcap_{v\in V}\{u\in S^{n-1}:\,<\!\!\!)(u,v)>\beta\}
\]
of unit normal directions that are more than $\beta$-far away from unit normals of the facets of $K$. 

\begin{cor} \label{coro:mainstabilityLP}
	Let a convex polytope  $K\subset \R^n$ with $rB^n\subset K$ for some $r>0$,  and an $(n-1)$-dimensional rectifiable weak barrier $B$ for $K$ be given. For any $\ee>0$ there is a constant $c=c(\ee,n,r,S(B))$ such that 
	\begin{equation}\label{eq:jbeta}
	S^*(B,J_\beta) \le \frac{c}{1-\cos\beta}\, \big(S(B)-\tfrac12 S(\partial K)\big)^{\frac{\llm{n+1}}{n(n+4)}-\ee}
	\end{equation}
	for all  $0<\beta<\pi/4$. 
\end{cor}
Hence, roughly speaking, most of the weak barrier is almost parallel to some facet of $K$. 
To illustrate this, consider the centered unit 
cube $K=[-1/2,1/2]^n \subset \R^n$ with $V$ consisting of the $n$ standard basis vectors and their antipodals. 
Under the assumptions of Corollary \ref{coro:mainstabilityLP}, the bound \eqref{eq:jbeta} states that the total surface area of the points of the barrier with a normal making an angle larger than  $\beta\in (0,\pi/4)$ with all coordinate directions, is small. 
For $n=2$, we have 
\begin{equation*}
	S^*(B,J_\beta) \le \frac c{1-\cos\beta}\, \delta^{\mk{\frac1{\llm{4}}-\ee}}, 
\end{equation*}
with $\delta=|B|-\tfrac12 |\partial K|$,
so almost all of the barrier has an orientation that deviates not more than $\beta$ from either horizontal or vertical. 
 It was shown in  \cite{steinerberger} by direct calculation for barriers $B$ of the unit cube that the bound
\begin{equation}\label{eq:steinerb} 
	S^*(B,J_\beta) \le\frac{2}{1-\cos\beta}\delta
\end{equation}
holds for all $0\le \beta\le \pi/4$ (the factor $2$ comes from a different normalization of the orientation measure in that paper). This  shows that the exponents in the above results \mk{can be improved. 
	
More generally, for large $n$, the exponent on the right hand side of \eqref{eq:mainStab} is essentially  of order $n^{-1}$. 
Since we do not expect that  barriers arbitrarily close to Jones' bound exist (we even know this for the majority of planar convex bodies as outlined above), it is futile to ask for the optimal exponent in \eqref{eq:mainStab}. However, the underlying purely geometric result in Theorem \ref{prop:KGMnew} below, has the same exponent and is derived from a stability result for the cosine transform (Proposition \ref{prop:KGM}). The exponent in the latter result is the best known to date, but whether it can be improved, is a valid open question, already mentioned in the semial paper \cite{BourgainLindenst} on the subject.}

The strong stability result \eqref{eq:steinerb} has been used by Pausinger and Kiderlen \cite{PausKid} to improve the best known lower bound for the total length of any barrier of the unit square in the plane. Kawamura \emph{et al.~}\cite{pach} have shown that this length is bounded from below by $2+2\cdot 10^{-5}$. Essentially by incorporating \eqref{eq:steinerb} in the line of arguments of \cite{pach}, it was shown in \cite{PausKid} that the bound can be improved to $2+6.3\cdot 10^{-5}$. This underlines the potential of the stability results, which might also be useful improving Jones' bound in higher dimensions. 
\medskip 

The paper is organized as follows. After recalling prerequisites from convex geometry and geometric measure theory in Section \ref{sect:prelim}, we present intermediate results and the proofs of the main results in Section \ref{sec.mainresults}. The last section puts the present results into perspective, and investigates the possibility to state a stability result of a different kind, which also takes into account position information of the parts of the barrier.

\section{Preliminaries}\label{sect:prelim}

Let $\cH^{n-1}$ denote \emph{$(n-1)$-dimensional Hausdorff measure} in $\R^n$. As in the introduction, we will write $S(A)=\cH^{n-1}(A)$ for the surface area of the Borel set $A\subset \R^n$. For $n=2$, the literature on barriers often writes 
$|A|=S(A)=\cH^1(A)$, a notation that we will adopt throughout the paper. 

For sets $A,B\subset\R^n$, $n\ge 2$, we will write $A+B=\{a+b:a\in A,b\in B\}$ for their \emph{Minkowski sum}. 
Define $rA=\{ra:a\in A\}$, $r\in \R$, and let $\partial A$ be the boundary of $A$. In $\R^n$, the \emph{Euclidean unit ball}, centered at the origin $o$,  is $B^n$ with volume $\kappa_n$, and the \emph{unit sphere} is denoted by $S^{n-1}=\partial B^n$. We have $\omega_n=\cH^{n-1}(S^{n-1})=n\kappa_n$. One should keep in mind that the standard notation $B^n$ for the unit ball is unrelated to the notation $B$ for a barrier. For a unit vector $u\in S^{n-1}$, $u^\perp$ is the hyperplane through the origin $o$ with unit normal vector $u$. We will write $A|u^\perp$ for the \emph{orthogonal projection} of $A$ on this hyperplane. \medskip

Recall that a \emph{convex body} is a non-empty compact convex subset of $\R^n$. 
We will often work with \emph{origin-symmetric} convex bodies $K$, i.e.~convex bodies such that $K=-K$. 
A convex body $K\subset \R^n$  is uniquely determined by its \emph{support function} $h(K,u)=\max_{x\in K}\langle u,x\rangle$, $u\in S^{n-1}$,  where 
the standard inner product $\langle\cdot,\cdot\rangle$ was used. These and other established notions from convex geometry are explained e.g.~in \cite{schneider}. 
Note that two convex bodies $K'$ and $K$ satisfy $K'\subset K$ if and only if $h(K',u)\le h(K',u)$ for all $u\in S^{n-1}$.
The length $h(K,u)+h(K,-u)$ of the projection of $K$ on a line with direction $u\in S^{n-1}$, is called the \emph{width of $K$ in that direction}, so
\[
w(K)=\frac1{\omega_n}\int_{S^{n-1}} \big(h(K,u)+h(K,-u)\big)\cH^{n-1}(\dd u)
\]
is the \emph{mean width}. If $K=[x,y]\subset \R^n$ is a line segment with endpoints $x,y\in \R^n$, direct calculation shows 
\begin{equation}
	\label{eq:widthOfSegment}
	w\big([x,y]\big)=\frac{2\kappa_{n-1}}{\omega_{n}}\|x-y\|,
\end{equation} 
and if $K\subset \R^2$ is a planar convex body with interior points, we have 
\begin{equation}
	\label{eq:meanwidth2D}
	w(K)=\tfrac1\pi |\partial K|.
\end{equation} 

For later use, we recall three simple results on balls contained in, and containing, a given convex body. The last statement makes use of the projection body $\Pi K$ of $K$, see \eqref{eq:projection_body}, below.
\begin{lem}\label{lem:0}
	\begin{itemize}
		\item[(i)] Let $K$ be a convex body containing the origin. Then   $K\subset R_0 B^n$ for $R_0=(\omega_n/(2\kappa_{n-1}))w(K)$.
		
		\item[(ii)] Let $K$ be a convex body 
		with  $r B^n\subset K$ for some $r>0$. Then  $K\subset R_0 B^n$ for $R_0=2^{n-1}\kappa_{n-2}^{-1}{r^{-(n-2)}}S(\partial K)$. 
		\item[(iii)] If $K$ is an origin-symmetric convex body such that $rB^n\subset \Pi K\subset RB^n$ for some $R\ge r>0$, then $r_0 B^n \subset K\subset R_0 B^n$ with
		\[
		R_0=\frac{\omega_n}r\left(\frac{R}{\kappa_{n-1}}\right)^{\frac n{n-1}},\qquad r_0=\frac{r}{2^{n-1}R_0^{n-2}}.
		\]
		
	\end{itemize}
	
\end{lem}
\begin{proof} 
	Statement (i) is a direct consequence of \eqref{eq:widthOfSegment}: whenever $x\in K$, we have $[o,x]\subset K$, so $w([o,x])\le w(K)$  gives an upper bound for $\|x\|$ and thus the claim. 
	As we are not interested in a sharp result in (ii), we argue with rough approximations. If $r B^n\subset K$ and $x\in K$, the convex hull of $rB^n\cup\{x\}$ is contained in $K$ by convexity. 
	This convex hull contains a cone with height $\|x\|$ and base radius $r$, which in turn contains a cylinder with height $\|x\|/2$ and base radius $r/2$. The surface area of this cylinder is larger than $\|x\|/2\cdot \kappa_{n-2}(r/2)^{n-2}$. As surface area is monotonic on convex bodies with respect to set inclusion, we get $S(\partial K)\ge 2^{1-n}\kappa_{n-2}r^{n-2}\|x\|$. Since $x\in K$ was arbitrary, the claim follows.
	
	With slightly different notation, (iii) is shown in the proof of \cite[Prop.~4.2]{GM}, starting from equation (16) in that paper. 
\end{proof}
\medskip

The distance between two convex bodies $K$ and $K'$ is usually measured by means of the $L_p$ metric $\delta_p$, $1\le p\le \infty$, given by 
\[
\delta_p(K,K')=\|h(K,\cdot)-h(K',\cdot)\|_p, 
\]
where $\|\cdot\|_p$ is the usual $L_p$ norm for functions on $S^{n-1}$ (with respect to the Hausdorff measure $\cH^{n-1}$, restricted to the Borel-subsets of $S^{n-1}$). For $p=\infty$ we obtain the \emph{Hausdorff metric} $\delta_\infty$. 

For a convex body $K\subset\R^n$ with interior points let $S(K,\cdot)$ be the surface area measure of $K$. Given a Borel set $A\subset S^{n-1}$, the number  $S(K,A)$ is the $\cH^{n-1}$-measure of all boundary points of $K$ that have an outer unit normal in $A$, see, e.g. \cite[Sect.~4.2]{schneider} for details. The total mass $S(K,S^{n-1})=S(\partial K)$ is the surface area of $\partial K$. 
The convex body $K$ is a convex polytope if and only if $S(K,\cdot)$ has finite support.
If $K$ has interior points, it is determined uniquely up to translation by its surface area measure. 
In this case,  $K$ is point-symmetric if and only if the measure  $S(K,\cdot)$ is even.
Due to Minkowski's existence theorem \cite[Thm.~8.2.2]{schneider}, a measure $\mu$
on the Borel sets of $S^{n-1}$ is the surface area measure of a convex body with interior points if and only if 
we have 
\begin{equation}
	\label{eq:minknecss1}
	\int_{S^{n-1}} u\,\mu(\dd u)=o,
\end{equation}
and $\mu$ is not concentrated on any great sub-sphere ($\mu(S^{n-1}\cap u^\perp)<\mu(S^{n-1})$ for all $u\in S^{n-1}$). This can be used to symmetrize a convex body. Indeed, let  $K\subset \R^n$ be a convex body with interior points. The unique origin-symmetric 
convex body $\nabla K$ such that 
$S(\nabla K,\cdot)=\tfrac12(S(K,\cdot)+S(-K,\cdot))$ is called the 
\emph{Blaschke body} of $K$, cf.~\cite[Def.~3.3.8]{Gardner}. Only for $n=2$ the linearity relation 
\begin{equation} \label{eq:S2}
	S(K,\cdot)+S(K',\cdot)=S(K+K',\cdot)
\end{equation}
 holds, so $\nabla K=\tfrac12(K+(-K))$
for convex bodies $K\subset\R^2$. \medskip 

We will make use of the inradius $$\rho(K)=\sup\{r\ge 0: \text{there is $x\in K$ with } x+rB^n\subset K\},$$
which can be bounded from above and below in terms of the volume $V_n(K)$ of $K$ and the surface area $S(\partial K)$ of its boundary:
\begin{equation}\label{eq:rhoKVol}
	\frac{V_n(K)}{S(\partial K)}\le \rho(K)\le n\frac{V_n(K)}{S(\partial K)}.
\end{equation} 
The first inequality follows from a much stronger result \cite[Coroll.~2 in Sect.~C of Part III]{osserman}, the second inequality is obtained from $$nV_n(K)=\int_{S^{n-1}} h(K,\cdot ) \dd S(K,\cdot)\ge \int_{S^{n-1}} \rho(K) \dd S(K,\cdot)=\rho(K) S(\partial K),$$ 
cf.~\cite[Thm.~5.1.7]{schneider}. 
From $\eqref{eq:rhoKVol}$ and the fact that $V_n(\nabla K)\ge V_n(K)$ (see \cite[Thm.~3.3.9]{Gardner}), we obtain 
\begin{equation}\label{eq:rhoNablaK}
	\rho(\nabla K)\ge \frac{V_n(\nabla K)}{S(\partial \nabla K)}\ge \frac{V_n(K)}{S(\partial K)}\ge \frac1n \rho(K).
\end{equation}
\medskip 

Projection areas of convex bodies on hyperplanes can be expressed in terms of the surface area measure, 
\begin{equation}
	\label{eq:projection_formula}
	\cH^{n-1}(K|u^\perp)=\int_{u^\perp} \1_{K\cap g_{x,u}\ne \emptyset}\,\cH^{n-1}(\dd x)=	\frac12 \int_{S^{n-1}}|\langle u,v\rangle| \, S(K,\dd v),
\end{equation}
$u\in S^{n-1}$, where $g_{x,u}=x+\myspan\{u\}$ is the line parallel to the vector $u$ through $x$.
As a function of $u$, the expression on the left hand side of \eqref{eq:projection_formula} is the projection function of $K$. It turns out that this function is the support function of a convex body, the \emph{projection body $\Pi K$of $K$}. By definition, we have 
\begin{equation}
	\label{eq:projection_body} 
	h(\Pi K,u)=\cH^{n-1}(K|u^\perp),\qquad u\in S^{n-1}. 
\end{equation}
Since $\Pi B^n=\kappa_{n-1} B^n$, we have
\begin{align}\label{eq:projBody}
	w(\Pi K)=\frac{2\kappa_{n-1}}{\omega_n}S(\partial K).
\end{align}

We will compare surface area measures and directional measures using the \emph{bounded Lipschitz metric} $d_{\mathrm{bL}}$, which is given by 
\[
d_{\mathrm{bL}}(\mu,\nu)=\sup\Big\{\left|\int_{S^{n-1}}f\dd(\mu-\nu)\right|: \|f\|_{\mathrm{bL}}\le 1\Big\},
\]
where $\|f\|_{\mathrm{bL}}=\sup_u|f(u)|+\sup_{u\ne v}\frac{|f(u)-f(v)|}{\|u-v\|}$ and $\mu$ and $\nu$ are finite measures on the Borel sets of $S^{n-1}$. 

Alternatively, one could use the \emph{Lévy–Prokhorov distance} 
given by 
\[
d_{\mathrm{LP}}(\mu,\nu)=\inf\{\ee>0: \mu(A)\le \nu(A^\ee)+\ee \text{ for all Borel } A\subset S^{n-1}\},
\]
where $A^\ee=\{u\in S^{n-1}: \inf_{v\in A} \|u-v\|<\ee\}$ consists of all unit vectors with Euclidean  distance  less than $\ee$ from $A$. 
If  $d_{\mathrm{bL}}(\mu,\nu)\le 1$ and $\mu(S^{n-1})>0$, we have 
\begin{equation*}
	d_{\mathrm{LP}}(\mu,\nu)
	\le (1+\sqrt{3+\mu(S^{n-1})})\,d_{\mathrm{bL}}(\mu,\nu)^{\frac12},
\end{equation*}
see \cite[Lem.~9.5]{GKM}, and this can be used to rephrase the stability results in this paper in terms of the Lévy-Prokhorov metric, albeit with a weaker exponent.  We will not rephrase our results with this metric here. 
\medskip 

The rest of this section recalls some basic notions from geometric measure theory, see e.g.~\cite{Morgan}. 
A set $B\subset \R^n$ is called $(\cH^{n-1},n-1)$-rectifiable if $\cH^{n-1}(B)<\infty$ and $\cH^{n-1}$-almost all of $B$ is contained in a countable union of the images of Lip\-schitz functions from $\R^{n-1}$ to $\R^n$. These sets are generalized surfaces in the sense of geometric measure theory. Following \cite{Morgan}, 
$(\cH^{n-1},n-1)$-rectifiable sets $B\subset \R^n$ that are  $\cH^{n-1}$-measurable are called \emph{$(n-1)$-dimensional rectifiable sets}. For instance, compact connected $(n-1)$-dimensional $C^1$-submanifolds, or the boundaries of convex bodies, are $(n-1)$-dimensional rectifiable sets.

If $B\subset \R^n$ is $(n-1)$-dimensional rectifiable, the cone $\mathrm{Tan}^{n-1}(B, x)$  of approximate tangent vectors of $B$  at $x$ forms an $(n-1)$-dimensional subspace in $\R^n$ for $\cH^{n-1}$-almost every $x\in B$. We will write $v_B(x)$ for a unit vector
perpendicular to this hyperplane and call it a \emph{unit normal vector} to $B$  at $x$. The unit normal vector is defined uniquely up to sign at $\cH^{n-1}$-almost all $x\in B$, so, writing $\delta_v$ for the Dirac probability measure supported by $\{v\}\subset S^{n-1}$, the finite measure
\[
S^*(B,\cdot)=\int_{B}(\delta_{v_B(x)}+\delta_{-v_B(x)})\,\cH^{n-1}(\dd x)
\]
is well-defined on the Borel sets of $S^{n-1}$. By definition, this measure is even, and we will call it the  \emph{orientation measure of $B$}.
Its total mass 
$$S^*(B)=S^*(B,S^{n-1})=2S(B)$$ 
is twice the surface area $S(B)=\cH^{n-1}(B)$ of $B$. Intuitively, if $B$ is a union of $(n-1)$-dimensional surface patches, $S^*(B,A)$ yields twice the surface area of all the patches with a normal in $A\subset S^{n-1}$. The factor $2$ is included here to take into account that a typical point of $B$ has two (antipodal) normal vectors. This also fits nicely to the definition of the related surface area measure for convex bodies: if, for instance, the convex body $K$ is contained in a hyperplane, it is $(n-1)$-dimensional rectifiable, and  $S^*(K,\cdot)=S(K,\cdot)$.

An analogue of \eqref{eq:projection_formula} can be obtained from an application of the coarea formula (see, e.g.~\cite[Thm.~2.93]{Ambrosio}), applied to the orthogonal projection on the hyperplane $u^\perp$. It reads
\begin{equation}
	\label{eq:coarea}
	\int_{u^\perp} \#(B\cap g_{x,u})\,\cH^{n-1}(\dd x)=
	\frac12 \int_{S^{n-1}}|\langle u,v\rangle| \, S^*(B,\dd v),
\end{equation}
$u\in S^{n-1}$. The left hand side of this equation is the total projection area of $B$ onto $u^\perp$ with multiplicities.

\section{Auxiliary results and proofs}\label{sec.mainresults}

We can now give a formal definition of weak barriers, again parametrizing the line $g_{x,u}=x+\myspan\{u\}$ by means of  $x\in u^\perp, u\in S^{n-1}$. 

\begin{defn}
	Assume that $B$ is an $(n-1)$-dimensional rectifiable subset of $\R^n$ and that $K$ is a convex body in $\R^n$. The set $B$ is called a \emph{weak barrier} for $K$ if 
	\begin{equation}\label{eq_defWeak}
		\int_{u^\perp} \#(B\cap g_{x,u})\,\cH^{n-1}(\dd x)\ge
		\int_{u^\perp} \1_{K\cap g_{x,u}\ne \emptyset}\,\cH^{n-1}(\dd x)=\cH^{n-1}(K|u^\perp)
	\end{equation}
	holds for all $u\in S^{n-1}$. 
\end{defn}
If $B$ is a barrier for $K$, then 
\begin{equation*}
	\int_{u^\perp} \1_{B\cap g_{x,u}\ne \emptyset}\,\cH^{n-1}(\dd x)\ge
	\int_{u^\perp} \1_{K\cap g_{x,u}\ne \emptyset}\,\cH^{n-1}(\dd x),
\end{equation*}
which is stronger than  \eqref{eq_defWeak}. Hence, if an $(n-1)$-dimensional rectifiable set is a  barrier for $K$, it is also a weak barrier for $K$. This is the main motivation for the naming, although descriptions of weak barriers as \emph{projection dominating sets} would be more mathematically saying. 

Assume that $B$ is an $(n-1)$-dimensional rectifiable weak barrier for a convex body $K$ with interior points. Since $\mu=S^*(B,\cdot)$ is even, it satisfies \eqref{eq:minknecss1}. 
In view of \eqref{eq:coarea} and \eqref{eq_defWeak} we have 
\[
\frac12 \int_{S^{n-1}}|\langle u,v\rangle| \, S^*(B,\dd v)\ge \cH^{n-1}(K|u^\perp)
\]
for all $u\in S^{n-1}$,
and the right hand side is positive, since $K$ has interior points. Hence, $\mu$ cannot be concentrated on any great sub-sphere, and Minkowski's existence theorem implies that there is a convex body with surface area measure $\mu$. This convex body is point-symmetric, as $\mu$ is even. Hence, the origin-symmetric convex body $\co(B)\subset \R^n$ with 
\begin{equation}
	\label{eq:Def_coB}
	S(\co(B),\cdot)=S^*(B,\cdot)
\end{equation}
is uniquely determined. It is called the \emph{convexification of $B$}. The two-dimensional construction for a straight barrier in the introduction is in accordance with the general definition. Indeed, if $B$ is the union of the line segments $s_1,\ldots,s_m\subset \R^2$, 
the convex set $\co(B)$ is up to translation equal to the Minkowski sum $s_1+\dots+s_m$, due to \eqref{eq:S2}. It is therefore a zonotope.

\begin{center}
	\begin{figure}[ht]
		
				\tdplotsetmaincoords{100}{70}
			
			\begin{tikzpicture}[tdplot_main_coords, scale=2, >=stealth]
				
				\coordinate (O) at (0,0,0);
				\filldraw[black] (O) circle (0.5pt) node[below left] {$o$};
				
				\coordinate (me1) at (0,-1.2,0);
				\coordinate (pe1) at (0,1.2,0);
				
				%
				\draw[thick, gray, dashed] (-0.5, -0.5, -0.5) -- (0.5, -0.5, -0.5);
				\draw[thick] (-0.5, -0.5, -0.5) -- (-0.5, 0.5, -0.5);
				\draw[thick] (-0.5, -0.5, -0.5) -- (-0.5, -0.5, 0.5);
				
				\def\R{sqrt(sqrt(3)/3.1415)/4*3}
				\def\Ru{\R/sqrt(2)}
				\def\Rv{\R/sqrt(6)}
				\filldraw[fill=gray!60, draw=black, opacity=0.8, thick] 
				plot[domain=0:360, samples=90, variable=\t] (
					{0.5 + \Ru*cos(\t) + \Rv*sin(\t)},
					{0.5 - \Ru*cos(\t) + \Rv*sin(\t)},
					{0.5             - 2*\Rv*sin(\t)}
				);
				
				\draw[thick, gray, dashed] (0.5, -0.5, -0.5) -- (0.5, 0.5, -0.5);
				\draw[thick] (-0.5, 0.5, -0.5) -- (0.5, 0.5, -0.5);
				
				\draw[thick] (-0.5, -0.5, 0.5) -- (0.5, -0.5, 0.5);
				\draw[thick] (-0.5, -0.5, 0.5) -- (-0.5, 0.5, 0.5);
				\draw[thick] (0.5, -0.5, 0.5) -- (0.5, 0.5, 0.5);
				\draw[thick] (-0.5, 0.5, 0.5) -- (0.5, 0.5, 0.5);
				
				\draw[thick, gray, dashed] (0.5, -0.5, -0.5) -- (0.5, -0.5, 0.5);
				\draw[thick] (-0.5, 0.5, -0.5) -- (-0.5, 0.5, 0.5);
				\draw[thick] (0.5, 0.5, -0.5) -- (0.5, 0.5, 0.5);
				
				\coordinate (X) at (0.5, 0.5, 0.5);
				\coordinate (Xp) at (0.5, 0.5, 0.48);
				\filldraw[black] (X) circle (0.5pt);
				\node[above] at (Xp) {$x$};
				\node[right] at (0.1, 0.7, 0.32) {$B_x$};
 
 			\draw[dotted,thin] (me1)  -- (pe1); 
				
			\end{tikzpicture}
			\qquad 			
				\tdplotsetmaincoords{70}{110}
			\begin{tikzpicture}[tdplot_main_coords, scale=2.3]
				
				\coordinate (A) at (1,0,0);   
				\coordinate (B) at (0,1,0);   
				\coordinate (C) at (-1,0,0);  
				\coordinate (D) at (0,-1,0);  
				\coordinate (Top) at (0,0,1); 
				\coordinate (Bot) at (0,0,-1);
				
				\draw[thick, gray!60, dashed] (C) -- (B);
				\draw[thick, gray!60, dashed] (C) -- (D);
				\draw[thick, gray!60, dashed] (C) -- (Top);
				\draw[thick, gray!60, dashed] (C) -- (Bot);
				
				\fill[gray!25, opacity=0.5] (Top) -- (B) -- (C) -- cycle;
				\fill[gray!25, opacity=0.5] (Top) -- (D) -- (C) -- cycle;
				\fill[gray!25, opacity=0.5] (Bot) -- (B) -- (C) -- cycle;
				\fill[gray!25, opacity=0.5] (Bot) -- (D) -- (C) -- cycle;
				
				\fill[gray!75, opacity=0.5] (Top) -- (A) -- (B) -- cycle;
				\fill[gray!75, opacity=0.5] (Top) -- (A) -- (D) -- cycle;
				\fill[gray!75, opacity=0.5] (Bot) -- (A) -- (B) -- cycle;
				\fill[gray!75, opacity=0.5] (Bot) -- (A) -- (D) -- cycle;

				\draw[thick, black!80] (A) -- (B);
				\draw[thick, black!80] (A) -- (D);
				\draw[thick, black!80] (Top) -- (A);
				\draw[thick, black!80] (Top) -- (B);
				\draw[thick, black!80] (Top) -- (D);
				\draw[thick, black!80] (Bot) -- (A);
				\draw[thick, black!80] (Bot) -- (B);
				\draw[thick, black!80] (Bot) -- (D);
		
			   \node (0.3,0.3,0) {$\co(B)$};
				
			\end{tikzpicture}

		\caption{
			\label{fig:last} \mk{Example illustrating the convexification concept in $\R^3$. Let $K$ be the centered unit cube. For any vertex $x$ of $K$, let $B_x$ be the two-dimensional disc in the plane  $x+x^\perp$ with center at $x$ and radius 
				$$r=\tfrac34\sqrt{\tfrac{\sqrt 3}{\pi}}\approx 0.56,$$ 
			as depicted on the left. Let $B=\bigcup_{\text{$x$ vertex of K}} B_x$ be the union of these 8 discs (not depicted).
			The convexification of $B$ is the  regular octahedron on the right. 
			The octahedron $\co(B)$ is not a zonotope as its boundary consists of triangles, cf.~\cite[Thm.~3.5.2]{schneider}. Hence, it cannot be written as a Minkowski sum of line segments in contrast to the 2D example in Fig.~\ref{bla}.
			Each triangle in the boundary of $\co(B)$ has area $2\pi r^2=9\sqrt3/8$, where the factor $2$ arises since opposing discs are in parallel planes. So, $\co(B)$ is the unique regular octahedron that circumscribes the centered unit cube. 
			In particular, $K\subset \co(B)$, which implies $\Pi K\subset \Pi\co(B)$. Thm.~\ref{thm:main} implies that $B$ is a weak barrier for $K$. It can be checked that $B$ is not a barrier for $K$ since the coordinate axes hit $K$ but not $B$, cf.~the dotted line in the left image. }
		}
	\end{figure}
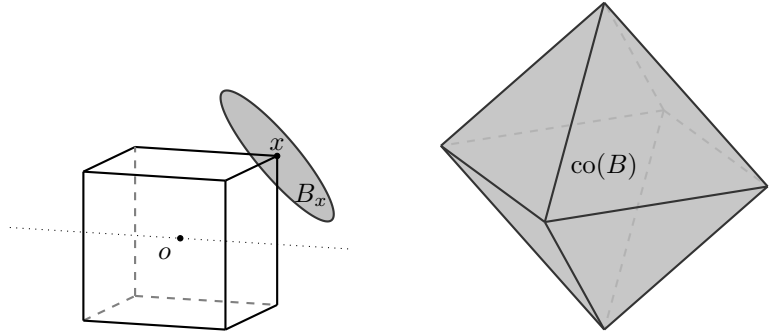
\end{center} 

Note, however, that in higher dimensions it is generally not enough to translate finitely many surface patches appropriately. For instance, if $B\subset \R^3$ is a union of finitely many two-dimensional discs, such that each disc has a normal parallel to an axis direction, then $\co(B)$ is an axis-parallel box. In dimensions $n\ge 3$, the convex body $\co(B)$ is not necessarily a zonotope (not even a zonoid), see Fig.~\ref{fig:last}. 

\begin{proof}[Proof of Thm.~\ref{thm:main}]
	Let a convex body $K\subset \R^n$ with interior points and an $(n-1)$-dimensional rectifiable set $B\subset\R^n$ be given. 
	
	In view of \eqref{eq:projection_formula}, \eqref{eq:coarea} and \eqref{eq:Def_coB} the defining inequality \eqref{eq_defWeak} of a weak barrier is equivalent to the statement that 
	\begin{align*}
		\cH^{n-1}(\co (B)|u^\perp)&=
		\frac12 \int_{S^{n-1}}|\langle u,v\rangle| \, S(\co (B),\dd v)\\&\ge
		\frac12 \int_{S^{n-1}}|\langle u,v\rangle| \, S(K,\dd v)=\cH^{n-1}(K|u^\perp)
	\end{align*}
	holds for all $u\in S^{n-1}$. By definition \eqref{eq:projection_body} of the projection body this is equivalent to $h(\Pi K,\cdot)\le h(\Pi (\co (B)),\cdot)$, which in turn is equivalent to  $\Pi K\subset \Pi (\co (B))$.

	In particular, since the mean width is monotone on convex bodies, we have $w(\Pi K) \le w\big(\Pi (\co (B))\big)$,
	and \eqref{eq:projBody} yields  
	\[
	S(\partial K) \le S(\partial\co (B))
	=2 S(B),
	\]
	confirming  Jones' bound \eqref{eq:jones}. 	
\end{proof}

We discuss the central condition of Theorem \ref{thm:main}, 
\begin{equation}\label{eq:Shepard}
	\Pi K\subset \Pi\big(\co (B)\big).
\end{equation} 
Since  $\Pi (K+x)=\Pi K$ for all $x\in \R^n$, condition \eqref{eq:Shepard}
will certainly not imply $K\subset \co (B)$. Using $\Pi K =\Pi (\nabla K)$, however, \eqref{eq:Shepard} is equivalent to 
$\Pi(\nabla  K)\subset \Pi\big(\co (B)\big)$, and now both,  $\nabla  K$ and $\co (B)$ are origin-symmetric. We have already seen that this condition implies   $\tfrac12 (K+(-K))=\nabla K\subset \co (B)$ when $n=2$. That this does not work in dimensions $n\ge 3$
can easily be seen by example. To fix ideas choose $n=3$ and let $B$ be such that $\co(B)=B^3$. Its hyperplane projections have area $\kappa_2=\llm{\pi}$. Let 
$K=\nabla K$ be a cylinder of length $s>2$ and radius $1/(2s)$. Its projections on  hyperplanes are contained in rectangles with side lengths $1/s$ and $s+1/s$ and their areas are at most $1+1/s^2<\llm{\pi}$, showing that \eqref{eq:Shepard} holds although $\nabla K\not \subset\co(B)$. 

\medskip 

The proof of Theorem \ref{thm:mainstabilityLP} requires some preparation.

\begin{lem}\label{lem:1}
	There is a constant $c_n>0$ such that 
	\[
	\delta_\infty(K',K)\le c_n w(K)^{1-\frac1n}\,(w(K)-w(K'))^{\frac1n}
	\]
	holds for all  origin-symmetric convex bodies	$K'\subset K$ in $\R^n$.
	
\end{lem}
\begin{proof} 
	We follow the proof of  \cite[Lemma 4]{steinerberger} for $n=2$. 
	Since $K$ is origin-symmetric, its support function $h(K,\cdot)$ coincides with half its width function and is thus Lipschitz with Lipschitz constant $(\diam K)/2$, cf.~\cite[Corollary 1.8.13]{schneider}. Hence, $f=h(K,\cdot)-h(K',\cdot)$ is Lipschitz with  constant 
	$$\frac{\diam K}2+\frac{\diam K'}2\le \diam K\le \tfrac{\omega_n}{2\kappa_{n-1}}w(K)=:c_L,$$
	cf.~Lemmma \ref{lem:0}.(i).  		
	If $f$ attains its maximal value $\delta_\infty=	\delta_\infty(K',K)$ at $u_*\in S^{n-1}$ then 
	\[
	f(u)\ge \max\{0,\delta_\infty-c_L\|u-u_*\|\}.
	\]
	Integrating both sides, using \eqref{eq:projBody} and cylindrical coordinates (cf.~\cite[Eq.(3.27)]{JensenKiderlen}), yields
	\begin{align*}
	\frac{\omega_n}2(w(K)-w(K'))&=	\int_{S^{n-1}} f(u)\,\cH^{n-1}(\dd u)
		\\&\ge \omega_{n-1}\int_{1-a}^1 (\delta_\infty-c_L\sqrt{2(1-t)})(1-t^2)^{\frac{n-3}2}\dd t,
	\end{align*}
	where we abbreviated $a=(\delta_\infty/c_L)^2/2$. 
	Due to \eqref{eq:widthOfSegment}, we have $c_L\ge 2\delta_\infty$, so $a\le 1/8$.
	A substitution gives 
	\begin{align}
		\frac{\omega_n}{2\omega_{n-1}}(w(K)-w(K'))&\ge \int_0^{a} (\delta_\infty-c_L\sqrt{2s})\big(s(2-s)\big)^{\frac{n-3}2}\dd s,\label{eq:integral}
	\end{align}
	For $n\ge 3$ this implies 
	\begin{align*}
		\frac{\omega_n}{2\omega_{n-1}}(w(K)-w(K'))&\ge \int_0^{a} (\delta_\infty-c_L\sqrt{2s})s^{\frac{n-3}2}\dd s
		\\&=\tfrac{2}{n(n-1)}a^{\frac{n-1}{2}}\delta_\infty\\&
		= \tfrac{1}{n(n-1)}2^{-\frac{n-3}2}c_L^{-(n-1)}\delta_\infty^n.
	\end{align*}
	In other words, 
	\begin{align*}
		\delta_\infty\le \sqrt[n]{\frac{n\omega_n}{\kappa_{n-1}}}\,\sqrt{2}^{1-\frac{5}{n}}c_L^{1-\frac1n}(w(K)-w(K'))^{\frac1n},
	\end{align*}
	yielding the claim for $n\ge3$. 
	For $n=2$ we have 
	\begin{align*}
	\frac\pi2(	w(K)-w(K'))&\ge \int_0^{a} (\delta_\infty-c_L\sqrt{2s})\big(s(2-s)\big)^{-\frac{1}2}\dd s
		\\&\ge \int_0^{a} (\delta_\infty-c_L\sqrt{2s})(2s)^{-\frac{1}2}\dd s=\tfrac1{2}c_L^{-1}\delta_\infty^2.
	\end{align*}
	This yields the claim for $n=2$. 
\end{proof}

\begin{lem}\label{lem3}
	Let $K'\subset K$ be two origin-symmetric convex bodies in $\R^n$. 
	Then 
	\[
	\delta_2(K',K)\le \llm{\sqrt{\frac{\omega_nc_n}{2}}}w(K)^{\llm{\frac{n-1}{2n}}}\,(w(K)-w(K'))^{\llm{\frac{n+1}{2n}}}
	\]
	with the constant $c_n$ from Lemma \ref{lem:1}. 
\end{lem}
\begin{proof}
	H\"older's inequality implies 
	\[
	\delta_2(K',K)\llm{^2}\le \delta_\infty(K',K)\delta_1(K',K),
	\]
	so Lemma  \ref{lem:1} and $\delta_1(K',K)=\frac{\omega_n}{2}(w(K)-w(K'))$ yield the claim. 
\end{proof}

We now quote \mk{a special case of the crucial stability result \cite[Thm.~5.1]{HS02} for multiplier transformations. Its proof is based on a clever use of the Poisson transform. In that result, more general
than the statement of our theorem, take $\mu=S(M,\cdot)-S(M',\cdot)$ and $\Phi(u,v) =|\langle u,v\rangle|$, so that 
according to \eqref{eq:projection_formula} and \eqref{eq:projection_body}, we have 
\begin{align*}
\tfrac12T_\Phi(\mu)&= 
\tfrac12\int_{S^{n-1}} |\langle u,\cdot\rangle|\, S(M,\dd u)-\tfrac12\int_{S^{n-1}} |\langle u,\cdot\rangle|\, S(M',\dd u)
\\&=h(\Pi M,\cdot)-h(\Pi M',\cdot). 
\end{align*}
Note that $\mu$ is an even measure. 
The parameter associated to this function $\Phi$ is $\beta =(n+2)/2$, and thus we get $\alpha=2/(n+4)$ in  \cite[Thm.~5.1]{HS02}. The proof  of \cite[Thm.~5.1]{HS02} is restricted to $n\ge3$, but has been extended to $n=2$ in 
\cite[Prop.~9.2]{GKM}, where the present line of arguments was outlined, but with an error in the calculation of $\alpha$, so the
upper bound $2/(n(n+4))$ of the exponent there should be replaced by the correct value $2/(n+4)$.}

\begin{prop}\label{prop:KGM}
	Let $M$ and $M'$ be origin-symmetric convex bodies in $\R^n$ with 
	\begin{equation}\label{eq:InnerOuterCircle}
		r_0 B^n\subset M,M'\subset R_0 B^n
	\end{equation}
	for some fixed $0<r_0\le R_0$. Then, for any $\ee>0$ there is a constant $c=c(\ee,n,r_0,R_0)$ with 
	\begin{equation}\label{eq:stabClassic}
		d_{\mathrm{bL}}\big(S(M,\cdot), S(M', \cdot)\big)\le c \,\delta_2(\Pi M,\Pi M')^{\mk{\frac2{n+4}}-\ee}. 
	\end{equation}	
\end{prop}


 \noindent 
We transfer this to a stability result adapted to the present problem. To describe the result intuitively, note that for two origin-symmetric  convex bodies $M,M'\subset\R^n$ we trivially  have 
\[
 |S(\partial M)-S(\partial M')| \le 
d_{\mathrm{bL}}\big(S(M,\cdot), S(M', \cdot),
\]
as one can choose $f\equiv 1$ in the definition of $d_{\mathrm{bL}}$. Hence, the difference of the surface areas of the two sets is controlled by the distance of their surface area measures. Under the additional assumption $\Pi M'\subset\Pi M$, the converse is also true, in the sense made precise in the following theorem. 

\begin{thm}\label{prop:KGMnew}
	 Let $M$ and $M'$ be origin-symmetric convex bodies in $\R^n$ with $rB^n\subset M'$ for some $r>0$, and let 
	 $\ee>0$. 
	 
	 If $\Pi M'\subset\Pi M$ there is a constant $c=c(\ee,n,r,S(\partial M))$ with 
	\begin{equation}\label{eq:exponent?}
		d_{\mathrm{bL}}\big(S(M,\cdot), S(M', \cdot)\big) \le c(S(\partial M)-S(\partial M'))^{\frac{\llm{n+1}}{n(n+4)}-\ee}.
	\end{equation}
\end{thm}
\begin{proof}	
	We intend to apply Proposition \ref{prop:KGM} to the convex bodies $M$ and $M'$ and therefore have to assure  \eqref{eq:InnerOuterCircle}. Using  $\Pi M'\subset\Pi M$ and \eqref{eq:projBody} (twice), we get 
	\[
		S(\partial M')=\frac{\omega_n}{2\kappa_{n-1}}w(\Pi M')\le \frac{\omega_n}{2\kappa_{n-1}}w(\Pi M)=S(\partial M).
	\]
	This, the assumptions of the theorem and Lemma \ref{lem:0}.(ii) thus imply 
	\begin{equation}
		\label{eq:ballsM'}
		r_0'B^n\subset M'\subset R_0' B^n
	\end{equation}  with constants 
	$r_0'=r$ and $R_0'=R_0'(n,r,S(\partial M))$.

		We now show similar inclusions for $M$, but start with $\Pi M$. We have 
	\begin{align*}
		\Pi M&\supset \Pi M'\supset \Pi\big( rB^n\big)
		= \kappa_{n-1} r^{n-1}B^n, 
	\end{align*}
	so Lemma \ref{lem:0}.(i), and \eqref{eq:projBody} allow to conclude $r_0''B^n\subset \Pi M\subset R_0'' B^n$ with   constants $r_0''=r_0''(n,r)$ and $R_0''=R_0''(n,S(\partial M))$. 
	Lemma \ref{lem:0}.(iii)
	implies 
	\begin{equation}
		\label{eq:ballsM}
		\tilde r_0B^n\subset  M\subset \tilde R_0 B^n
	\end{equation}   
	with constants only depending on $n,r,S(\partial M)$. Concluding, equations \eqref{eq:ballsM'} and \eqref{eq:ballsM}
	show that \eqref{eq:InnerOuterCircle} is satisfied
	with constants  $r_0$ and $R_0$ only depending on $n,r$ and $S(\partial M)$.

	Thus, Proposition \ref{prop:KGM}  gives a constant $c_1=c_1(\ee, n, r, S(\partial M))$ such that \eqref{eq:stabClassic} holds. 
	Lemma \ref{lem3} can be  applied to $\Pi M$ and $\Pi M'$, since $\Pi M'\subset \Pi M$. This gives
	\begin{align*}
		\delta_2(\Pi M,\Pi M')
		&\le \tilde c_2(w(\Pi  M)-w(\Pi M'))^{\llm{\frac{n+1}{2n}}}\\
		&= c_2\big(S(\partial M)-S(\partial M')\big)^{\llm{\frac{n+1}{2n}}},
	\end{align*}
	where $\tilde c_2$ and $c_2$ depend only on $n$ and $S(\partial M)$, 
	and \eqref{eq:projBody} was used. Inserting this last bound into \eqref{eq:stabClassic} yields the assertion.
\end{proof} 


\begin{proof}[Proof of Theorem \ref{thm:mainstabilityLP}]

Let the assumptions of Theorem \ref{thm:mainstabilityLP} be satisfied. \mk{In particular, we have $rB^n\subset K$. Due to the bound in \eqref{eq:rhoNablaK}, there is a ball of radius  $r/n>0$ contained in  $M'=\nabla  K$. (Such a ball exists even with the additional constraint that it is centered at the origin, since $M'$ is origin-symmetric.) 
The convex body $M=\co(B)$ satisfies $\Pi M'\subset \Pi M$ due to Theorem \ref{thm:main}. 

	Theorem  \ref{prop:KGMnew} now guarantees the existence  of a constant $c=c(\ee,n,r,S(B))$ such that \eqref{eq:exponent?} holds. 
	Inserting the identities  $S(\partial M')= S(\partial \nabla K )=S(\partial  K)$ and $S(\partial M)= S(\partial \co(B) )=2S(B)$ 
	 gives the assertion \eqref{eq:mainStab}.
	}
\end{proof}
\bigskip

\begin{proof}[Proof of Corollary \ref{coro:mainstabilityLP}]
	Let the assumptions of  Corollary \ref{coro:mainstabilityLP} be satisfied and let $V=\{u_1, \ldots,u_m\}$ be the set outer unit normal vectors of the facets of $\nabla K$. We will write $t^+=\max \{t,0\}$, $t\in \R$. Define first $m$ Lipschitz functions in the following way: for $i=1,\ldots,m$, the functions $g_i:S^{n-1}\to [0,1]$, given by 
	\[
	g_i(u)=\frac{(\langle u_i,u\rangle-\cos \beta)^+}{1-\cos\beta},\quad u\in S^{n-1},
	\]
	satisfy $\|g_i\|_\infty=g_i(u_i)=1$, and are Lipschitz functions with Lipschitz constant $(1-\cos\beta)^{-1}$. 
	The function $g_i$ vanishes outside the relative-open spherical cap with center $u_i$ and opening angle $\beta$. Hence, $f:S^{n-1}\to [0,1]$ with 
	\[
	f(u)=1-\max_{i=1}^m g_i(u)
	\]
	is identical $1$ if restricted to the complement $J_\beta$ of the union of these $m$ caps. Thus, 
	\[
	\int_{S^{n-1}} f(u) S^*(B,\dd u)\ge \int_{J_\beta} f(u) S^*(B,\dd u)=S^*(B,J_\beta). 
	\]
	In addition, 
	$f(u_i)=0$, $i=1,\ldots,m$, so $\int_{S^{n-1}} f(u) S(\nabla K,\dd u)=0$, and 
	\begin{align}\nonumber
		S^*(B,J_\beta)&\le \int_{S^{n-1}} f \dd(S^*(B,\cdot)-S(\nabla K,\cdot))
		\\&\label{eq:1Coro}
		\le \|f\|_{\mathrm{bL}}	d_{\mathrm{bL}}\big(S(\nabla K,\cdot),S^*(B,\cdot) \big). 
	\end{align}
	Since  $\|f\|_\infty\le1$ and $\|f\|_L\le (1-\cos\beta)^{-1}$, we have   
	\[
	\|f\|_{\mathrm{bL}}\le1+(1-\cos\beta)^{-1}\le 2(1-\cos\beta)^{-1}. 
	\]
	This can be inserted into \eqref{eq:1Coro}, and the assertion follows by an application of Theorem \ref{thm:mainstabilityLP}. 
\end{proof}

\section{Concluding remarks}
We have shown a stability result for barriers with surface area close to Jones' bound in arbitrary dimension. In addition, we have seen that these stability results even hold for weak barriers, in other words, that they are only depending on the orientation of the barrier, ignoring the actual positions of the constituting parts.
These stability results pave the way for improving Jones' bound also in the higher dimensional setting, as documented by the application of stability results in the plane in \cite{PausKid}.

One might, however, alternatively aim for another stability result, which would even  take the position information of the parts of the barrier into account. This is now explained. Consider a hypothetical barrier $B$ for a full-dimensional  convex body $K$ with equality in Jones' bound ($S(B)=S(\partial K)/2$). It is not difficult to see that then almost all lines hitting $K$ must hit $B$ exactly once, and this is only possible if almost all of $B$ is contained in a hyperplane. Hence, most lines hitting $K$ that are parallel to this hyperplane but not included in it, cannot hit the set $B$. Thus $B$ cannot be a barrier, and we have confirmed that Jones' bound cannot be satisfied with equality.  

What is missing is a stability version of this observation: assuming that the surface area of the barrier $B$ is very close to $S(\partial K)/2$, is it then true that most of $B$ can be partitioned into two subsets, one of which lying `very close' to a hyperplane $E$, and the other one being `very far away' from $E$? Clearly, this is a property that requires the exact positions of the parts of $B$, and hence cannot hold for weak barriers of $K$. Historical attempts to improve Jones' bound show that properties of positions of the boundary are much more difficult to exploit than those depending on orientation information only (weak barriers). 
We therefore expect that this second kind of stability result is even more difficult to obtain than the present one. 

\section*{Declarations}
The author has no relevant financial or non-financial interests to disclose.
No funds, grants, or other support was received.



\end{document}